\documentclass{elsarticle}

\usepackage{amsfonts, amsthm, amssymb, amsmath}
\usepackage{enumerate}
\usepackage{dsfont}
\usepackage{mathtools}

\usepackage{tikz}
\usetikzlibrary{arrows}

\usepackage{color}
\usepackage[
pdfauthor={CSY},
pdftitle={Regular colorings and factors},
pdfstartview=XYZ,
bookmarks=true,
colorlinks=true,
linkcolor=blue,
urlcolor=blue,
citecolor=blue,
bookmarks=true,
linktocpage=true,
hyperindex=true
]{hyperref}

\newtheorem{theo}{Theorem}[section]
\newtheorem{lemma}[theo]{Lemma}

\newtheorem{corollary}[theo]{Corollary}

\theoremstyle{definition}

\newtheorem{conj}[theo]{Conjecture}

\newtheorem{ques}[theo]{Question}
\newtheorem{question}[theo]{Question}

\theoremstyle{remark}
\newtheorem{remark}[theo]{Remark}

\newcommand{\dotcup}{\mathbin{\dot\cup}}

\begin{document}

\begin{frontmatter}

\title{Regular colorings and factors of regular graphs\tnoteref{GRWC}}
\tnotetext[GRWC]{This research was partially supported by NSF Grant 1500662 ``The 2015 Rocky Mountain - Great Plains Graduate Research Workshop in Combinatorics.''}

\author[UIUC]{Anton Bernshteyn\fnref{Antonsk}}
\author[UM]{Omid Khormali}
\author[ISU]{Ryan R. Martin\fnref{Ryansk}}
\author[KIT]{Jonathan Rollin}
\author[QU]{Danny Rorabaugh}
\author[VU]{Songling Shan}
\author[UNL]{Andrew J. Uzzell}
\address[UIUC]{Department of Mathematics, University of Illinois at Urbana-Champaign}
\address[UM]{Department of Mathematical Sciences, University of Montana}
\address[ISU]{Department of Mathematics, Iowa State University}
\address[KIT]{Department of Mathematics, Karlsruhe Institute of Technology}
\address[QU]{Department of Mathematics and Statistics, Queen's University}
\address[VU]{Department of Mathematics, Vanderbilt University}
\address[UNL]{Department of Mathematics, University of Nebraska--Lincoln}
\fntext[Antonsk]{Research of this author is supported by the Illinois Distinguished Fellowship.}
\fntext[Ryansk]{This work was supported by a grant from the Simons Foundation (\#353292, Ryan Martin).}

\begin{abstract}
An {\it $(r-1,1)$-coloring} of an $r$-regular graph $G$ is an edge coloring such that each vertex is incident to $r-1$ edges of one color and $1$ edge of a different color.
In this paper, we completely characterize all $4$-regular pseudographs (graphs that may contain parallel edges and loops) which do not have a $(3,1)$-coloring.
An {\it $\{r-1,1\}$-factor} of an $r$-regular graph is a spanning subgraph in which each vertex has degree either $r-1$ or $1$.
We prove various conditions that that must hold for any vertex-minimal $5$-regular pseudographs without $(4,1)$-colorings or without $\{4,1\}$-factors.
Finally, for each $r\geq 6$ we construct graphs that are not $(r-1,1)$-colorable and, more generally, are not $(r-t,t)$-colorable for small $t$.
\end{abstract}


\end{frontmatter}

\section{Introduction}

A graph with no loops or multiple edges is called \emph{simple}; 
a graph in which both multiple edges and loops are allowed is
called a \emph{pseudograph}.
Unless specified otherwise, the word ``graph'' in this paper is reserved for pseudographs.
All (pseudo)graphs considered here are undirected and finite.
Note that we count a loop twice in the degree of a vertex.

The famous Berge--Sauer conjecture asserts that every $4$-regular simple graph contains a $3$-regular subgraph~\cite{MR0411988}.
This conjecture was settled by Tashkinov in 1982~\cite{Tashkinov-4-regular}.  
In fact, he proved that every connected $4$-regular pseudograph with either at most~two pairs of multiple edges and no loops or at most~one pair of multiple edges and at most~one loop contains a $3$-regular subgraph.  
Observe that this cannot hold for all $4$-regular pseudographs, because the graph consisting of a single vertex with two loops contains no $3$-regular subgraph.
The following question remains open.

\begin{ques}\label{QUE:3-reg-subgraph}
Which $4$-regular pseudographs contain $3$-regular subgraphs?
\end{ques}

Note that in 1988, Tashkinov~\cite{Tashkinov2} classified the values of $t$ and $r$ for which every $r$-regular pseudograph contains a $t$-regular subgraph. 
Beyond finding regular subgraphs in regular graphs, finding factors---that is, regular spanning subgraphs---in regular graphs
is also of special interest.  
As early as~$1891$, Petersen~\cite{Petersen} studied the existence of factors in regular graphs.
Since then numerous results on factors have appeared---see, for example, \cite{AKBook, BSW, CG, Plu}.
The concept of factors can be generalized as follows:  
for any set of integers~$S$, an {\it $S$-factor} of a graph is a spanning subgraph in which the degree of each vertex is in $S$.
Several authors~\cite{AK, AR, LWY} have recently studied $\{a,b\}$-factors in $r$-regular graphs with $a+b=r$.
In particular, Akbari and Kano~\cite{AK} made the following conjecture:

\begin{conj}\label{CONJ:a-b-factor}
If $r$ is odd and $0 \leq t \leq r$, then every $r$-regular graph has an $\{r-t,t\}$-factor.
\end{conj}

However, Axenovich and Rollin~\cite{AR} disproved this conjecture. 
The following theorem summarizes what is known about $\{r-t,t\}$-factors of $r$-regular graphs.
(Note that although intended for simple graphs, the result of Petersen~\cite{Petersen} applies to pseudographs as well.)

\begin{theo}\label{THM:values}
Let $t$ and $r$ be positive integers with $t\le \frac{r}{2}$.
\begin{enumerate}[(a)]
	\item When $r$ is even:
	\begin{itemize}
		\item If $t$ is even, then every $r$-regular graph has a $t$-factor, and thus has an $\{r-t,t\}$-factor (Petersen~\cite{Petersen}).
		\item Every $r$-regular graph of even order has an $\left\{\frac{r}{2}+1,\frac{r}{2}-1\right\}$-factor (Lu, Wang, and Yu~\cite{LWY}).
		\item If $t$ is odd and $t \leq \frac{r}{2} - 2$, then there exists an $r$-regular graph of even order that has no $\{r-t,t\}$-factor (\cite{LWY}). 
		\item If $t$ is odd, then trivially, no $r$-regular graph of odd order has an $\{r-t,t\}$-factor.
	\end{itemize}
	\item When $r$ is odd and $r \geq 5$:
	\begin{itemize}
		\item If $t$ is even, then every $r$-regular graph has an $\{r-t,t\}$-factor (Akbari and Kano~\cite{AK}).
		\item If $t$ is odd and $\frac{r}{3} \leq t$, then every $r$-regular graph has an $\{r-t,t\}$-factor (\cite{AK}).
		\item If $t$ is odd and $(t+1)(t+2) \leq r$, then there exists an $r$-regular graph that has no $\{r-t,t\}$-factor (Axenovich and Rollin~\cite{AR}).
	\end{itemize}
	\item Every $3$-regular graph has a $\{2,1\}$-factor (Tutte~\cite{Tutte-3regular}).
	\end{enumerate}
\end{theo}

The first case of Conjecture~\ref{CONJ:a-b-factor} that Theorem~\ref{THM:values} does not address is when $r = 5$ and $t = 1$.  As we will give much of our attention to this case, we restate it separately.

\begin{conj}\label{CONJ:4-1-factor}
Every $5$-regular graph has a $\{4,1\}$-factor.
\end{conj}

An {\it $(r-t,t)$-coloring} of an $r$-regular graph $G$ is an edge-coloring (with at least~two colors) such that each vertex is incident to $r-t$ edges of one color and $t$ edges of a different color.  An \emph{ordered $(r-t,t)$-coloring} of $G$ is an $(r-t,t)$-coloring using integers as colors 
such that each vertex is incident to $r-t$ edges of some color $i$ and $t$ edges of some color $j$ with $i < j$.  Bernshteyn~\cite{B} introduced $(3, 1)$-colorings as an approach to answering Question~\ref{QUE:3-reg-subgraph}.  The advantage of working with $(3,1)$-colorings is that this notion is ``global'' (i.e., there is a condition at each vertex),
	while the presence of a $3$-regular subgraph is a ``local'' notion (a large $4$-regular graph can contain a small $3$-regular subgraph).  Bernshteyn proved the following.

\begin{theo}[Bernshteyn~\cite{B}] \label{3-regular-ordered-coloring}
		A connected $4$-regular graph contains a $3$-regular subgraph if and only if it admits an ordered $(3,1)$-coloring.
\end{theo}

We observe that the notion of an $(r - t, t)$-coloring of an $r$-regular graph generalizes that of an $\{r-t,t\}$-factor, because $\{r-t,t\}$-factors correspond to $(r-t,t)$-colorings that use exactly~two colors.  (In an $r$-regular graph with $0<t<r$, $t$-factors correspond to \emph{ordered} $(r-t,t)$-colorings that use exactly~two colors.)  Thus, $(r-t, t)$-colorings provide a common approach to attacking Question~\ref{QUE:3-reg-subgraph} and Conjecture~\ref{CONJ:4-1-factor}.  This leads us to ask whether the following weaker version of Conjecture~\ref{CONJ:4-1-factor} holds.

\begin{ques}\label{QUE:4-1-coloring}
Does every $5$-regular graph have a $(4,1)$-coloring?
\end{ques}

For $r \geq 6$, the answer to the analogue of Question~\ref{QUE:4-1-coloring} for $(r-1, 1)$-colorings is negative (see Section~\ref{6reg}).

Similarly, Theorem~\ref{3-regular-ordered-coloring} motivates the following weaker version of Question~\ref{QUE:3-reg-subgraph}.

\begin{ques}\label{QUE:3-1-coloring}
Which $4$-regular graphs have $(3, 1)$-colorings?
\end{ques}

The arrows in Figure~\ref{Relation} indicate the relationships among $t$-factors, $\{r-t,t\}$-factors,
ordered $(r-t,t)$-colorings, $(r-t,t)$-colorings, and $t$-regular subgraphs of $r$-regular graphs. 

 \begin{figure}[ht]
	\centering
	
	\begin{tikzpicture}
		\draw (0,0) node {$G$ has a $t$-factor.};
		\draw[double,->] (0,.2) -- (-2.2,.8);
		\draw[double,->] (0,.2) -- (2.2,.8);
		\draw (-2.8,1) node {$G$ has an $\{r-t,t\}$-factor.};
		\draw (2.8,1) node {$G$ has an ordered $(r-t,t)$-coloring.};
		\draw[double,->] (2.8,1.2) -- (2.8,1.8);
		\draw[double,->] (1.1,1.2) -- (-1.1,1.8);
		\draw[double,->] (-2.8,1.2) -- (-2.8,1.8);
		\draw (2.8,2) node {$G$ has a $t$-regular subgraph.};
		\draw (-2.8,2) node {$G$ has an $(r-t,t)$-coloring.};
	\end{tikzpicture}
\label{Relation}
	\caption{Implications that hold for every $r$-regular graph $G$ and for all integers $0< t < r$.}	
\end{figure}
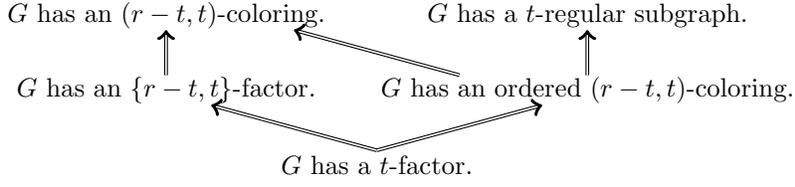

Now we are ready to describe our main results.
First, in Section~\ref{4reg}, we characterize all $4$-regular graphs which are not $(3,1)$-colorable, which settles Question~\ref{QUE:3-1-coloring}.  Because the statement of the result requires additional definitions, we postpone it until then (see Theorem~\ref{THM:4-regular-no-3-1}).  Then, in Section~\ref{5reg}, we make progress toward settling Conjecture~\ref{CONJ:4-1-factor} and Question~\ref{QUE:4-1-coloring} by proving several conditions on vertex-minimal $5$-regular graphs without $(4,1)$-colorings and $\{4,1\}$-factors.
Finally, in Section~\ref{6reg}, we construct relevant examples of $r$-regular graphs for $r \geq 6$ and various $t$: some with no $(r-t,t)$-coloring, others with an $(r-t,t)$-coloring but no $\{r-t,t\}$-factor.

\section{\texorpdfstring{$(3,1)$-colorings in $4$}{(3, 1)-colorings in 4}-regular graphs} \label{4reg}

In this section, we characterize $4$-regular graphs that do not admit $(3, 1)$-colorings.

Let us first establish some terminology.    Let $G_1$ and $G_2$ be vertex-disjoint graphs with edges $e_1=u_1v_1 \in E(G_1)$ and~$e_2=u_2v_2 \in E(G_2)$. 
    The \emph{edge adhesion} of $G_1$ and $G_2$ at $e_1$ and $e_2$ is the graph $G = (G_1, e_1) + (G_2, e_2)$ obtained by subdividing edges $e_1$ and~$e_2$ and identifying the two new vertices. 
    (See Figure~\ref{fig:EdgeAdhere}.) 
    That is,
    \begin{eqnarray*}
        V(G) & = & V(G_1) \dotcup V(G_2) \dotcup \{w\}; \\
        E(G) & = & (E(G_1)\setminus \{e_1\}) \dotcup (E(G_2)\setminus\{e_2\}) \dotcup \{u_1w,v_1w,u_2w,v_2w\}.
    \end{eqnarray*}
    
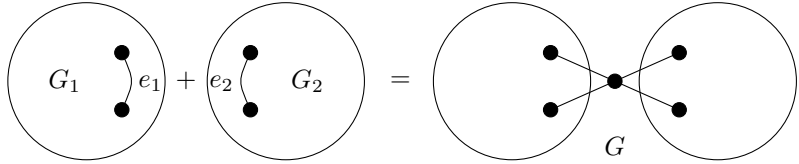
\begin{figure}[ht]
    \centering
    \begin{tikzpicture}[scale=.95]
    	\foreach \x in {-.9,.9}{
    		\draw (1.55*\x,0) circle(11mm);
    		\foreach \y in {-.4,.4}{ \filldraw (\x,\y) circle(1mm); };
    		\draw (\x,-.4) ..controls (.8*\x,0) .. (\x,.4);
    	}
        \draw (0,0) node{$e_1 \; + \; e_2$};
        \draw (-1.7,0) node{$G_1$};
        \draw (1.7,0) node{$G_2$};
        
        \draw (3,0) node{$=$};
        
        \begin{scope}[xshift=6cm]
	    	\foreach \x in {-.9,.9}{
	    		\draw (1.6*\x,0) circle(11mm);
	    		\foreach \y in {-.4,.4}{ 
	    			\filldraw (\x,\y) circle(1mm);
	    			\draw (\x,\y) -- (0,0);
			 };
	    	}
	    	\filldraw (0,0) circle(1mm);
	    	\draw (0,-.9) node{$G$}; 
        \end{scope}
    \end{tikzpicture}
    \caption{Edge adhesion of two graphs, $G=(G_1,e_1)+(G_2,e_2)$.} \label{fig:EdgeAdhere}
\end{figure}
    
    The \emph{adhesion of a loop} to graph $H$ at edge $e = uv \in E(H)$ is the graph $H' = (H,e) + O$ obtained by subdividing $e$ and adding a loop at the new vertex.  (See Figure~\ref{fig:AddLoop}.) That is,
    \begin{eqnarray*}
        V(H') & = & V(H) \dotcup \{x\}; \\
        E(H') & = & (E(H)\setminus \{e\})\dotcup \{ux,vx,xx\}.
    \end{eqnarray*}

\begin{figure}[ht]
    \centering
    \begin{tikzpicture}[scale=.93]
        \draw (.4,.1) circle(11mm);
        \filldraw (0,0) circle(1mm);
        \filldraw (.6,.8) circle(1mm);
        \draw (0,0) .. controls (.38,.34) .. (.6,.8);
        \draw (.44,.41)  node[right,below]{$e$};
        \draw (.4,-.7) node{$H$};
        
        \draw (2.25,0) node{$+\;O = $};
        
        \begin{scope}[xshift=3.6cm]
            \draw (.5,.1) circle(11mm);
            \filldraw (0,0) circle(1mm) --(.38,.34) circle(1mm) --(.6,.8) circle(1mm);
            \draw (.38,.34) .. controls (.38,-.46) and (1.18,.34) .. (.38,.34);
      	  \draw (.5,-.7) node{$H'$};
        \end{scope}
    \end{tikzpicture}
    \caption{Adhesion of a loop at an edge, $H'=(H,e)+O$.} \label{fig:AddLoop}
\end{figure}
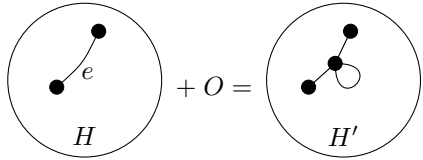

Let $C$ be a (simple) cycle. 
A {\it double cycle} is obtained from $C$ by doubling each edge. 
We say a double cycle is even (respectively, odd) if it has an even (respectively, odd) number of vertices. 
(See Figure \ref{fig:DoubleCycles}.)

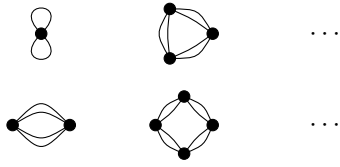
\begin{figure}[ht]
    \centering
    \begin{tikzpicture}[scale=.76]
        \filldraw (0,0) circle(1mm);
        \draw (0,0) .. controls (-.6,.6) and (.6,.6) .. (0,0);
        \draw (0,0) .. controls (-.6,-.6) and (.6,-.6) .. (0,0);
        \begin{scope}[xshift=25mm]
            \foreach \x in {0,120,240}{
                \filldraw (\x:.5) circle(1mm);
                \draw (\x:.5) .. controls (\x+60:.5) .. (\x+120:.5);
                \draw (\x:.5) .. controls (\x+60:.3) .. (\x+120:.5);
            };
        \end{scope}
        \draw (5,0) node {$\cdots$};
        \begin{scope}[yshift=-16mm]
            \foreach \x in {0,180}{
                \filldraw (\x:.5) circle(1mm);
                \draw (\x:.5) .. controls (\x+90:.5) .. (\x+180:.5);
                \draw (\x:.5) .. controls (\x+90:.3) .. (\x+180:.5);
            };
            \begin{scope}[xshift=25mm]
                \foreach \x in {0,90,180,270}{
                    \filldraw (\x:.5) circle(1mm);
                    \draw (\x:.5) .. controls (\x+45:.5) .. (\x+90:.5);
                    \draw (\x:.5) .. controls (\x+45:.3) .. (\x+90:.5);
                };
            \end{scope}
            \draw (5,0) node {$\cdots$};
        \end{scope}
    \end{tikzpicture}
    \caption{Double cycles (odd on top, even on bottom).} \label{fig:DoubleCycles}
\end{figure}

Clearly, double cycles and graphs resulting from edge adhesion of two $4$-regular graphs or from the adhesion of a loop to a $4$-regular graph are all $4$-regular.  
We are now ready to give the main result of this section.

\begin{theo}\label{THM:4-regular-no-3-1}
	A connected $4$-regular graph is not $(3,1)$-colorable if and only if it can be constructed from odd double cycles via a sequence of edge adhesions.
\end{theo}

\begin{remark}\label{re:4-regular-no-3-1-inductive}
Theorem~\ref{THM:4-regular-no-3-1} naturally lends itself to a proof by induction.  In particular, an equivalent statement is that a connected $4$-regular graph is not $(3,1)$-colorable if and only if it is an odd double cycle or obtained from two $4$-regular, non-$(3,1)$-colorable graphs by a sequence of edge adhesions.
\end{remark}

Before we prove Theorem~\ref{THM:4-regular-no-3-1}, we need to develop a few lemmas.

\begin{lemma} \label{lemma:Cycles}
    A double cycle with $n \geq 1$ vertices is $(3,1)$-colorable if and only if $n$ is even.
\end{lemma}
\begin{proof}Even double cycles have perfect matchings and are thus $(3,1)$-colorable.

Assume that there is a $(3,1)$-coloring $c$ of an odd double cycle $G$.
Let $G'$ denote the cycle obtained by removing one of the parallel edges between any two adjacent vertices in $G$.
Color an edge in $G'$ red if the corresponding edges in $G$ are of the same color under $c$ and blue otherwise.
Observe that the edges incident to any vertex in $G'$ are of different colors, since $c$ is a $(3,1)$-coloring of~$G$.
This is a contradiction since $G'$ is an odd cycle.
\end{proof}

\begin{lemma}[Bernshteyn~\cite{B}] \label{lemma:Edge}
	If $G$ is a $4$-regular graph and there exists a non-double edge $uv$ in $G$ with $u \neq v$ such that $G-\{u,v\}$ is connected, then $G$ is $(3,1)$-colorable.
\end{lemma}

\begin{lemma}[Bernshteyn~\cite{B}] \label{lemma:AddLoop}
    If $G$ is a $4$-regular graph and $G' = (G,e)+O$ for some edge $e \in E(G)$, then either $G$ or $G'$ has a $3$-regular subgraph.
\end{lemma}

\begin{lemma} \label{lemma:JoinAtLoop}
    Let $G_1$ and $G_2$ be $(3,1)$-colorable $4$-regular graphs and let $G_2$ have a loop $vv$.
    Construct $G$ by subdividing an edge~$uw$ in $G_1$, identifying the new vertex with $v$, and removing the loop~$vv$, so
    \begin{eqnarray*}
        V(G) & = & V(G_1) \dotcup V(G_2); \\
        E(G) & = & (E(G_1)\setminus \{uw\}) \dotcup (E(G_2)\setminus\{vv\}) \dotcup \{uv,wv\}.
    \end{eqnarray*}
   (See Figure~\ref{fig:JoinAtLoop}.) Then $G$ is $(3,1)$-colorable.
\end{lemma}

\begin{figure}[ht]
    \centering
    \begin{tikzpicture}
    	\foreach \x in {-1.2,.9}{
    		\draw (1.55*\x,0) circle(11mm);
    		\foreach \y in {-.4,.4}{ \filldraw (\x,\y) circle(1mm); };
    	};
        \draw (.9,.4) node[right]{$u$} ..controls (.7,0) .. (.9,-.4) node[right]{$w$};
        \draw (-2,0) node{$G_1$};
        \draw (1.7,0) node{$G_2$};
        \foreach \y in {-.4,.4}{ \draw (-1.2,\y) -- (-.4,0); };
        \draw (-.2,0) circle(2mm);
        \draw (-.46,0) node[above]{$v$};
        \filldraw (-.4,0) circle(1mm);
        
        \draw (3,0) node{$\longrightarrow$};
        
        \begin{scope}[xshift=6cm]
	    	\foreach \x in {-.9,.9}{
	    		\draw (1.6*\x,0) circle(11mm);
	    		\foreach \y in {-.4,.4}{ 
	    			\filldraw (\x,\y) circle(1mm);
	    			\draw (\x,\y) -- (0,0);
			 };
	    	};
	    	\draw (0,-.9) node{$G$}; 
	    	\draw (.9,.4) node[right] {$u$};
	    	\draw (.9,-.4) node[right] {$w$};
	    	\filldraw (0,0) circle(1mm) node[above]{$v$};
        \end{scope}
    \end{tikzpicture}
    \caption{Joining $G_2$ to $G_1$ at a loop, as in Lemma~\ref{lemma:JoinAtLoop}.} \label{fig:JoinAtLoop}
\end{figure}
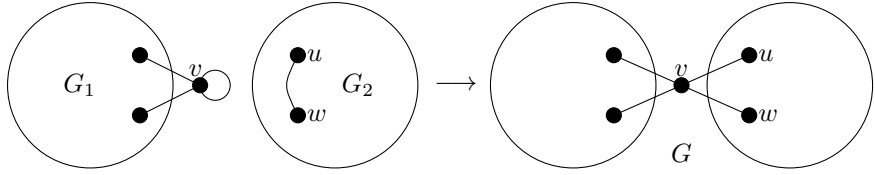

\begin{proof}
    Fix $(3,1)$-colorings $c_i$ of $G_i$ for $i \in \{1,2\}$.
    Note that $v$ in $G_2$ is incident to only~one loop and that the two non-loop edges incident to $v$ have different colors under $c_2$.
    Without loss of generality, assume that $c_1(uw)$ is equal to the color of one of the non-loop edges incident to $v$.
    Therefore the colorings $c_1$ and~$c_2$ extend to a $(3,1)$-coloring of $G$ by coloring the edges $uv$ and~$uw$ with color~$c_1(uw)$.    
\end{proof}

\begin{corollary}[to Lemmas~\ref{lemma:AddLoop},~\ref{lemma:JoinAtLoop}] \label{cor:Adhes}
	Suppose exactly~one of the connected $4$-regular graphs $G_1$ and~$G_2$ is $(3,1)$-colorable.
	Then for any $e_1 \in E(G_1)$ and $e_2 \in E(G_2)$, $(G_1,e_1)+(G_2,e_2)$ is $(3,1)$-colorable.
\end{corollary}

\proof
Without loss of generality, we assume that $G_1$ is $(3,1)$-colorable and $G_2$ is not.
Let $e_1 \in E(G_1)$ and $e_2 \in E(G_2)$.  
By Theorem~\ref{3-regular-ordered-coloring} and Lemma~\ref{lemma:AddLoop}, the graph $G'_2=(G_2,e_2)+O$ is $(3,1)$-colorable. 
Applying Lemma~\ref{lemma:JoinAtLoop} to $G_1$ and $G'_2$, we see that $(G_1,e_1)+(G_2,e_2)$ is $(3,1)$-colorable.
\qed

\begin{lemma} \label{lemma:2conn}
	Let $G$ be a $4$-regular graph that is not $(3,1)$-colorable.
	If $G$ has a non-double, non-loop edge, then $G$ is not $2$-connected.
\end{lemma}

\begin{proof}
    Let $uv$ be a non-double, non-loop edge, and suppose for contradiction that $G$ is $2$-connected.  By Lemma~\ref{lemma:Edge}, since $G$ is not $(3,1)$-colorable, $G' = G - \{u,v\}$ is disconnected.
    Since $G$ is $2$-connected, neither $u$ nor $v$ is a cut-vertex.
    Therefore, every component of $G'$ must contain at least one vertex from $N_G(u)$ and at least one vertex from $N_G(v)$.
    Since the sum of the degrees of the vertices must be even in each component, the $4$-regularity of $G$ implies that each component of $G'$ must have an even number of vertices from $N_G(u) \cup N_G(v)$.
    Let $N_G(u) \setminus \{v\} = \{u_1,u_2,u_3\}$ and $N_G(v) \setminus \{u\} = \{v_1,v_2,v_3\}$.
    Without loss of generality, $G'$ is the disjoint union of a component~$G_1$ containing $u_1$ and $v_1$ and a subgraph~$G_2$ (of one or two components) containing $u_2$, $u_3$, $v_2$, and $v_3$.

	Let $G'_1 = (G_1 + u_1v_1, u_1v_1) + O$ and $G'_2 = ((G-G_1) + uv, uv) + O$. 
	(See Figure~\ref{fig:2Conn}.)
	That is,
	\begin{eqnarray*}
		V(G'_1) & = & V(G_1) \dotcup \{w_1\};\\
		E(G'_1) & = & E(G_1) \dotcup \{u_1w_1,v_1w_1,w_1w_1 \};\\
		V(G'_2) & = & V(G_2) \dotcup \{u,v,w_2\};\\
		E(G'_2) & = & E(G_2) \dotcup \{uu_2,uu_3,uv,vv_2,vv_3,uw_2,vw_2,w_2w_2 \}.
	\end{eqnarray*}
	
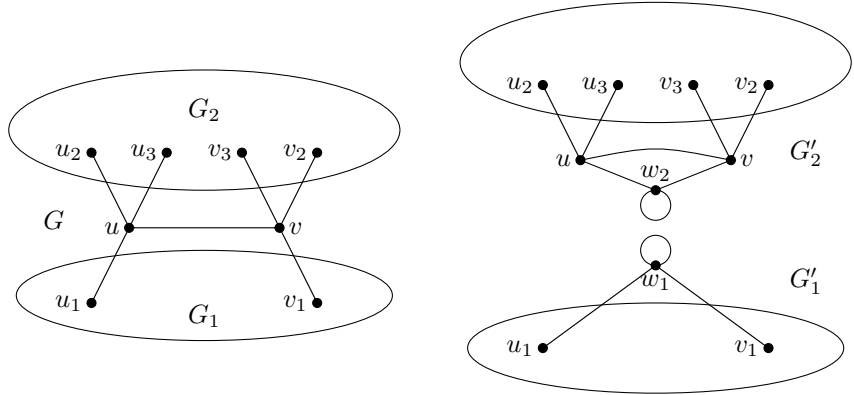
\begin{figure}[ht]
	\centering

	\begin{tikzpicture}
		\draw (-1,.1) node {$G$};
		\filldraw (0,0) circle(.06) node[left]{$u$} -- (2,0) circle(.06) node[right]{$v$};
		\filldraw (-.5,-1) circle(.06) node[left]{$u_1$} -- (0,0);
		\filldraw (2.5,-1) circle(.06) node[left]{$v_1$} -- (2,0);
		\draw (1,-.9) ellipse(25mm and 6mm) node[below]{$G_1$};
		\filldraw (-.5,1) circle(.06) node[left]{$u_2$} -- (0,0);
		\filldraw (2.5,1) circle(.06) node[left]{$v_2$} -- (2,0);
		\filldraw (.5,1) circle(.06) node[left]{$u_3$} -- (0,0);
		\filldraw (1.5,1) circle(.06) node[left]{$v_3$} -- (2,0);
		\draw (1,1.3) ellipse(26mm and 8mm) node[above]{$G_2$};

		\begin{scope}[xshift=6cm, yshift=.9cm]
			\draw (0,0) .. controls (1,.2) .. (2,0) -- (1,-.4) -- (0,0);
			\filldraw (1,-.4) circle(.06) node[above]{$w_2$};
			\draw (1,-.6) circle(2mm);
			\filldraw (-.5,1) circle(.06) node[left]{$u_2$} -- (0,0) circle(.06) node[left]{$u$};
			\filldraw (2.5,1) circle(.06) node[left]{$v_2$} -- (2,0) circle(.06) node[right]{$v$};
			\filldraw (.5,1) circle(.06) node[left]{$u_3$} -- (0,0);
			\filldraw (1.5,1) circle(.06) node[left]{$v_3$} -- (2,0);
			\draw (1,1.3) ellipse(26mm and 8mm);
			\draw (3,0.1) node {$G'_2$};
		\end{scope}

		\begin{scope}[xshift=6cm, yshift=-.6cm]
			\filldraw (-.5,-1) circle(.06) node[left]{$u_1$} -- (1,.1) circle(.06);
			\filldraw (2.5,-1) circle(.06) node[left]{$v_1$} -- (1,.1) node[below]{$w_1$};
			\draw (1,.3) circle (2mm);
			\draw (1,-1) ellipse(25mm and 6mm);
			\draw (3,-.1) node{$G'_1$};
		\end{scope}
	\end{tikzpicture}

	\caption{Splitting a $2$-connected graph into two $(3,1)$-colorable graphs, from the proof of Lemma~\ref{lemma:2conn}.} \label{fig:2Conn}
\end{figure}

	By the assumption of $2$-connectedness, the vertex $u_1$ is not a cut-vertex of $G$, so $u_1 \neq v_1$ and $G'_1 - \{u_1,w_1\}$ is connected.
	Thus by Lemma~\ref{lemma:Edge}, $G'_1$ is $(3,1)$-colorable.
	Likewise, $G'_2 - \{u,w_2\}$ is connected, so $G'_2$ is $(3,1)$-colorable.
	Select $(3,1)$-coloring $c_i$ of $G'_i$ for $i \in \{1,2\}$.
	Note that because of the loops, $c_1(u_1w_1) \neq c_1(v_1w_1)$ and $c_2(uw_2) \neq c_2(vw_2)$.
	We can assume that $c_1(u_1w_1) = c_2(uw_2)$ and $c_1(v_1w_1) = c_2(vw_2)$.
	Therefore, the colorings $c_1$ and $c_2$ easily extend to a $(3,1)$-coloring $c$ of~$G$, 
	which is a contradiction.
\end{proof}

\begin{lemma}\label{lemma:split}
	Let $G$ be a connected $4$-regular graph that is not $2$-connected. Then $G = (G_1, e_1) + (G_2, e_2)$ for some $4$-regular graphs $G_1$, $G_2$ and edges $e_1 \in E(G_1)$, $e_2 \in E(G_2)$.
\end{lemma}

\begin{proof}
	Indeed, let $w \in V(G)$ be a cut-vertex. Now the lemma is implied by the following observation. Since the number of vertices with odd degrees in a graph is always even, $G-w$ consists of exactly two components and each of these components receives exactly two of the edges incident to $w$.
\end{proof}

\proof[Proof of Theorem~\ref{THM:4-regular-no-3-1}]
Consider $4$-regular graphs $G_1$ and $G_2$ and edges $e_1$ in $G_1$, $e_2$ in $G_2$.
Any $(3,1)$-coloring of $(G_1,e_1)+(G_2,e_2)$ yields a $(3,1)$-coloring of $G_1$ or $G_2$, since the edges obtained by subdividing $e_1$ or $e_2$ are of the same color.
Therefore every graph that is obtained from odd double cycles via edge adhesion is not $(3,1)$-colorable due to Lemma~\ref{lemma:Cycles}.

Now let $G$ be a connected $4$-regular graph that is not $(3,1)$-colorable. 
We use induction on $|V(G)|$ to prove that $G$ is constructed from odd double cycles via edge adhesion. 
If $|V(G)|=1$, then $G$ is a double cycle of one vertex and the theorem trivially holds. 
Assume that $|V(G)|\ge 2$.
We may also assume that $G$ contains a non-double edge.  
Otherwise, if every edge is double, then $G$ is a double cycle, and by Lemma~\ref{lemma:Cycles}, $G$  is an odd double cycle,
and thus we are done.  

If each non-double edge is a loop, then one can easily check that $G$ is not $2$-connected.
If $G$ has a non-double non-loop edge, Lemma~\ref{lemma:2conn} implies that it is not $2$-connected. 
By Lemma~\ref{lemma:split}, $G = (G_1, e_1) + (G_2, e_2)$ for some $4$-regular graphs $G_1$, $G_2$ and edges $e_1 \in E(G_1)$, $e_2 \in E(G_2)$. 
Corollary~\ref{cor:Adhes} implies that either both $G_1$ and $G_2$ are $(3,1)$-colorable or neither of them is $(3,1)$-colorable. 
In the latter case, by the inductive hypothesis, we are done.
 
Assume that both $G_1$ and $G_2$ are $(3,1)$-colorable. 
Let $G'_1 = (G_1,e_1) + O$ and observe that
$G$ is obtained from $G'_1$ and $G_2$ as in the statement of Lemma~\ref{lemma:JoinAtLoop}. 
Since $G_2$ is $(3,1)$-colorable, but $G$ is not, Lemma~\ref{lemma:JoinAtLoop} implies that $G'_1$ is not $(3,1)$-colorable. 
Therefore, by the inductive hypothesis, $G'_1$ is obtained from odd double cycles via edge adhesion. 
Since $G'_1$ contains a loop and at least two vertices, it is not a double cycle. 
Thus, $G'_1 = (G'_{11}, e'_{11}) + (G'_{12}, e'_{12})$, where both $G'_{11}$ and $G'_{12}$ are not $(3,1)$-colorable. 
 Note that, without loss of generality, $G'_{11}$ does not contain the subdivided edge $e_1$, and so $G = (G'_{11}, e'_{11}) + (H, f)$ for some graph $H$ and edge $f$ in $H$.
 Since both $G$ and $G'_{11}$ are not $(3,1)$-colorable, neither is $H$ by Corollary~\ref{cor:Adhes}. 
 We have shown that $G$ is obtained from two graphs that are not $(3,1)$-colorable via edge adhesion, and so the inductive step is complete.
\qed
 
\section{\texorpdfstring{$(4,1)$-colorings and $\{4,1\}$-factors in $5$}{(4,1)-colorings and \{4,1\}-factors in 5}-regular graphs} \label{5reg}

In this section, we make progress toward settling Conjecture~\ref{CONJ:4-1-factor} and Question~\ref{QUE:4-1-coloring}.  
In particular, we show that if $G$ is a vertex-minimal counterexample to Conjecture~\ref{CONJ:4-1-factor}, then $G$ must satisfy a large number of structural conditions.  We show that similar conditions must hold for any vertex-minimal graph that gives a negative answer to Question~\ref{QUE:4-1-coloring}.

A set $S$ of edges of a connected graph $G$ is called an \emph{edge cut} if $G-S$ is disconnected.
An edge cut $S$ is \emph{minimal} provided $G-(S\setminus\{e\})$ is connected for each edge $e\in S$.
An edge cut of size $1$ is called a \emph{bridge}.
Note that a minimal edge cut does not contain loops.

Most of the following results are obtained using reductions to smaller graphs.
We also use a corollary of Tutte's $1$-Factor Theorem.

\begin{theo}[Tutte~\cite{tutte-1-factor}]
	A graph $G$ has a $1$-factor if and only if the number of connected components of $G-S$ of odd order is at most~$|S|$ for every vertex set $S \subseteq V(G)$.
\end{theo}

\begin{corollary} \label{corollary:2kEdgeConnected}
	Every $2k$-edge-connected $(2k+1)$-regular graph has a $1$-factor.
\end{corollary}

In Section~\ref{se:no-4-1-coloring}, we prove our results about $(4, 1)$-colorings.  In Section~\ref{se:no-4-1-factor}, we prove our results about $\{4, 1\}$-factors.  

\subsection{\texorpdfstring{$5$}{5}-regular graphs without \texorpdfstring{$(4, 1)$}{(4, 1)}-colorings}\label{se:no-4-1-coloring}

We begin by showing that a vertex-minimal $5$-regular graph with no $(4,1)$-coloring must satisfy several connectivity conditions.
An edge-coloring $c$ of $G$ \emph{extends} an edge-coloring $c'$ of $G'$ if $c(e)=c'(e)$ for all $e\in E(G)\cap E(G')$.

\begin{theo}\label{minNo41coloring}
 Let $G$ be a vertex-minimal $5$-regular graph without a $(4,1)$-coloring.
 \begin{enumerate}[(a)]
  \item $G$ is connected. \label{connected}
  \item $G$ is not $4$-edge-connected, i.e., contains an edge cut on $3$ edges. \label{not4EdgeConnected}
  \item $G$ has no minimal edge cut of size $2$. \label{no2cut}
  \item $G$ does not have two bridges. \label{no2bridges}
  \item Each bridge in $G$ has (precisely) one endpoint incident to two loops. \label{special1cut}
  \item The edges of any minimal edge cut of size $3$ in $G$ have a vertex in common, and this vertex is incident to a loop. \label{special3cut}
 \end{enumerate}
\end{theo}

\begin{proof}
(\ref{connected}) This follows from vertex-minimality.

\medskip

(\ref{not4EdgeConnected}) This is a consequence of Corollary~\ref{corollary:2kEdgeConnected} with $k=2$. 

\medskip

(\ref{no2cut}) Assume $\{uv,wx\}$ is a minimal edge cut, so $G - \{uv,wx\}$ is disconnected, but $G-uv$ and $G-wx$ are both connected.
  Then $G-\{uv,wx\}$ has precisely~two components, $G_1$ and~$G_2$, and, without loss of generality, $u$,~$w \in V(G_1)$ and $v$,~$x \in V(G_2)$.
   We obtain $5$-regular graphs $G'_1 = G_1 + uw$ and $G'_2 = G_2 + vx$ by adding a new edge (possibly a loop or parallel edge) to each component.
  By the assumption of vertex-minimality, both graphs $G'_1$ and $G'_2$ have $(4,1)$-colorings.
  Consider such colorings $c_i$ of $G'_i$ for $i \in \{1,2\}$ such that $c_1(uw)=c_2(vx)=1$.
  Note that all edges of $E(G)\setminus\{uv,wx\}$ are contained in exactly one of $G'_1$ or $G'_2$.
  So we obtain a $(4,1)$-coloring of $G$ by coloring $uv$ and $wx$ with color~$1$ and all other edges according to $c_1$ and $c_2$, a contradiction. 

\medskip

(\ref{no2bridges}) Assume $uv$ and $wx$ are bridges in $G$.
  Then $G-\{uv,wx\}$ has three components.
  Without loss of generality, assume that $u$ and $w$ are contained in the same component.
  We obtain two $5$-regular graphs by adding the edges $uw$ and $vx$  (possibly loops or parallel edges).
  The proof proceeds exactly as in (\ref{no2cut}). 

\medskip

(\ref{special1cut})  If there is a bridge with both endpoints incident to two loops, then there are no other edges and the graph is easily $(4,1)$-colorable.
  Assume that there is a bridge $uv$ with each endpoint incident to at most one loop.
  Then $G-uv$ has two components $G_1$ and $G_2$, each with at least~$2$ vertices.
  We obtain a $5$-regular graph from $G_1$ (respectively, $G_2$) by adding a new vertex incident to two loops and to $u$ (respectively, to $v$).
  Both graphs have $(4,1)$-colorings by assumption of vertex-minimality.
  Much as before, we obtain a $(4,1)$-coloring of $G$ by choosing the same color for the new edges incident to $u$ and $v$, a contradiction. 

\medskip

(\ref{special3cut}) Consider distinct edges $uv$, $wx$ and $yz$ forming a minimal edge cut of size~$3$.
  First observe that a vertex which is incident to all three edges is incident to a loop due to statements (\ref{no2cut}) and (\ref{no2bridges}) of this theorem.
  Thus assume that there is no such vertex.
  Removing the three edges from $G$ yields exactly two components $G_1$ and~$G_2$, each with at least $2$ vertices.
  Without loss of generality assume $u$, $w$ and $y$ are in $G_1$ and $v$, $x$ and $z$ are in $G_2$.

  Let $\mathcal{H}_i$ denote the set of all $5$-regular graphs that contain $G_i$ as a subgraph and have one more vertex than $G_i$, $i \in \{1,2\}$.
  Note that $\mathcal{H}_i\neq\emptyset$.
  We consider each $H\in\mathcal{H}_i$ with a fixed copy~$K=K(H)$ of $G_i$ and call edges in $E(H)\setminus E(K)$ \emph{new} if they are incident to vertices of~$K$.
  By assumption all graphs in $\mathcal{H}_1\cup \mathcal{H}_2$ are $(4,1)$-colorable.
  Assume first that, for all $i\in\{1,2\}$, there is a graph $H\in\mathcal{H}_i$ having a $(4,1)$-coloring where all ($2$ or $3$) new edges are of the same color.
  Then, much as before, we obtain a $(4,1)$-coloring of $G$, a contradiction.
  
  So, assume that for any graph $H\in\mathcal{H}_2$ and for any $(4,1)$-coloring of~$H$ the new edges in $H$ are not all of the same color.
  Consider a $(4,1)$-coloring~$c_1$ of the graph in $\mathcal{H}_1$ obtained from $G_1$ by adding a new vertex $p$ incident to one loop and connected to $u$, $w$ and $y$ by new edges.
  Without loss of generality assume that $c_1(up)=c_1(wp)\neq c_1(yp)$.
  Further consider a $(4,1)$-coloring $c_2$ of the graph in $\mathcal{H}_2$ obtained from $G_2$ by adding by adding a new vertex $q$ incident to two loops and edges $vx$ and~$qz$.
  Then $c_2(vx)\neq c_2(qz)$ by assumption.
  Therefore we obtain a $(4,1)$-coloring of~$G$ from $c_1$ and $c_2$ as before, a contradiction. \qedhere
\end{proof}

Now we prove a number of conditions involving loops, parallel edges, or forbidden subgraphs (see Figure~\ref{fig:41coloringForbidden}).

\begin{theo} \label{le:5regDoubleEdgesAndLoops}
Let $G$ be a vertex-minimal $5$-regular graph without a $(4, 1)$-coloring.
\begin{enumerate}[(a)]
  \item $G$ does not contain a $4$-regular subgraph with at least~$2$ vertices. \label{no4regular}
  \item $G$ does not have $3$ parallel edges. \label{3parallel}
  \item $G$ does not contain a path of length three consisting of double edges. \label{noDouble3Path}
  \item No vertex of~$G$ that has a loop is incident to a double edge. \label{noLoopDoubleIncidence}
  \item No vertices with loops are adjacent. \label{noLoopAdjacence}
  \item $G$ contains at least $5$ loops. \label{5LoopMin}
  \item There do not exist $u_1$, $u_2$, $u_3$, $v_1$, $v_2$, $v_3 \in V(G)$ such that the $u_i$ have loops and such that for each $i$ and $j$, $u_i$ is adjacent to $v_j$ (that is, there is no $K_{3,3}$ with one loop on each vertex of one side of the vertex partition). \label{noK33withLoops}
  \item No vertex is adjacent to more than~$3$ vertices with loops. \label{no4LoopNeighbors}
  \item No $4$-vertex subgraph of $G$ has $8$ or more edges. \label{Forbidden4}
\end{enumerate}
\end{theo}

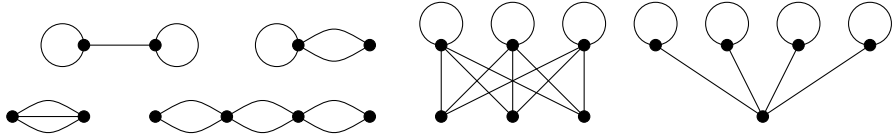
\begin{figure}[ht]
	\centering

	\begin{tikzpicture}[scale=.95]
	\foreach \x in {0,1}{ \filldraw (\x,0) circle(.08); };
	\foreach \y in {-.3,0,.3}{ \draw (0,0) .. controls(.5,\y) .. (1,0); };

    \begin{scope}[xshift=2cm]
    	\foreach \x in {0,1,2}{
    		\filldraw (\x,0) circle(.08);
    		\draw (\x,0) .. controls(\x+.5,.3) .. (\x+1,0);
    		\draw (\x,0) .. controls(\x+.5,-.3) .. (\x+1,0);
    	};
    	\filldraw (3,0) circle(.08);
    \end{scope}
    
    \begin{scope}[xshift=4cm, yshift=1cm]
    	\draw (-.3,0) circle(.3);
    	\foreach \x in {0,1}{ \filldraw (\x,0) circle(.08); };
    	\foreach \y in {-.3,.3}{ \draw (0,0) .. controls(.5,\y) .. (1,0); };
    \end{scope}
    
    \begin{scope}[xshift=1.5cm, yshift=1cm]
    	\draw (-.5,0) -- (.5,0);
    	\foreach \x in {-.5,.5}{ \filldraw (\x,0) circle(.08); };
    	\foreach \x in {-.8,.8}{ \draw (\x,0) circle(.3); };
    \end{scope}
    
    \begin{scope}[xshift=6cm]
    	\foreach \x in {0,1,2}{
    		\foreach \y in {0,1}{ \filldraw (\x,\y) circle(.08); };
    		\draw (\x,1.3) circle(.3);
    		\foreach \z in {0,1,2} { \draw (\x,0) -- (\z,1); };
    	};
    \end{scope}
    
    \begin{scope}[xshift=9cm]
    	\foreach \x in {0,1,2,3}{
    		\draw (1.5,0) -- (\x,1);
    		\filldraw (\x,1) circle(.08);
    		\draw (\x,1.3) circle(.3);
    	};
    	\filldraw (1.5,0) circle(.08);
    \end{scope}

\end{tikzpicture}

	\caption{From Theorem~\ref{le:5regDoubleEdgesAndLoops}~(\ref{3parallel},~\ref{noDouble3Path},~\ref{noLoopDoubleIncidence},~\ref{noLoopAdjacence},~\ref{noK33withLoops},~\ref{no4LoopNeighbors}), forbidden subgraphs in a vertex-minimal $5$-regular graph with no $(4,1)$-coloring.}\label{fig:41coloringForbidden}
\end{figure}

\begin{proof}
(\ref{no4regular}) Suppose for contradiction that $G$ has a $4$-regular subgraph~$H$ with at least~$2$ vertices.
  Let $F$ denote the set of edges $uv$ in $G$ with $u\in V(H)$ and $v\not\in V(H)$.
  We obtain a $5$-regular graph~$G'$ from $G$ by removing the vertices of~$H$ 
  and adding some \emph{new} edges between the vertices of degree less than~$5$ and, if there is an odd number of such vertices, one new vertex with two loops.  
  Let $F'$ denote the set of new edges in $G'$, except for the loops incident to the new vertex, if such exists.
  
  By vertex-minimality, $G'$ has a $(4,1)$-coloring $c$.
  We extend this to a coloring of~$G$ as follows.
  Assign a color $k$ not used by $c$ to all edges in $H$ and a color different from $k$ to all edges in $E(G)\setminus E(H)$ having both endpoints in $H$.
  Each edge in $F$ shares a vertex with least one edge in $F'$.
  Consider an injective map $f:F \to F'$ such that $e$ and $f(e)$ have a common vertex for all $e\in F$.
  Then color each $e\in F$ with color $c(f(e))$.
  This coloring is a $(4,1)$-coloring of $G$, a contradiction. 

\medskip

(\ref{3parallel}) Assume that there are at least three edges between vertices $u$ and $v$.
  Let $F$ denote the set of edges incident to $u$ or $v$ but not both.
  Observe that $G$ has at least~3 vertices, as the $2$-vertex $5$-regular graphs are easily $(4,1)$-colorable.
  Consider the $5$-regular graph $G'$ obtained from $G$ by removing $u$ and $v$ and adding a matching between the (remaining) neighborhood of $u$ and the (remaining) neighborhood of $v$, possibly creating parallel edges and loops.
  By assumption $G'$ has a $(4,1)$-coloring $c$.
  We extend this to a coloring of~$G$ by coloring the edges in $F$ with the colors of the corresponding new edges under $c$.
  Then $u$ and $v$ are either both incident to edges of the same color only, or both incident to an edge of one color and an edge of a second color.
  In either case we can color the parallel edges between $u$ and $v$ such that we obtain a $(4,1)$-coloring of $G$, a contradiction. 

\medskip

(\ref{noDouble3Path})  Let $v_1$, $v_2$, $v_3$, and $v_4$ denote the vertices of a double path in $G$.  
Let $u_2$ be the other neighbor of $v_2$ and $u_3$ the other neighbor of $v_3$.  
We assume that $u_2\neq v_3$ and $u_3\neq v_2$ due to part~(\ref{3parallel}) of this Theorem.
Remove $v_2$ and $v_3$ from $G$, add two edges between $v_1$ and $v_4$, and add an edge (possibly a loop or multiple edge) between $u_2$ and $u_3$.  
Let $G'$ denote the resulting graph, which, by hypothesis, has a $(4, 1)$-coloring.  
We consider several cases (see Figure~\ref{fig:noDoub}).

\begin{figure}[ht]
\centering
\begin{tikzpicture}[scale=.5]

\begin{scope}
	\draw (-1,2.5) .. controls (0,2.2) .. (1,2.5);
	\filldraw (-3,1) circle(.14);
	\foreach \y in {1.4, .6} { \draw (-3,1) .. controls (0,\y) .. (3,1); };
	\filldraw (3,1) circle(.14);
	
	\draw [->,very thick] (0,.2) -- (0,-.8);
	
	\draw (-1,-3) .. controls (0, -2.7) .. (1, -3);
	\draw[dashed,red] (-1,-3) .. controls (0, -3.3) .. (1, -3);
	\foreach \x in {-3, 1} {
		\foreach \y in {-2.7, -3.3} { \draw (\x,-3) .. controls (\x+1,\y) .. (\x+2,-3); };
		\filldraw (\x,-3) circle(.14);
	};
	\foreach \x in {-1,1} {
		\filldraw (\x,2.5) circle(.14);
		\filldraw (\x,-1.5) circle(.14) -- (\x, -3);
	};
	\filldraw (3,-3) circle(.14);
	\filldraw (-1,-3) circle(.14);
\end{scope}

\begin{scope}[xshift=105mm]
	\draw[dashed,red] (-1,2.5) .. controls (0,2.2) .. (1,2.5);
	\foreach \y in {1.4, .6} { \draw (-3,1) .. controls (0,\y) .. (3,1); };
	\filldraw (-3,1) circle(.14);
	\filldraw (3,1) circle(.14);
	
	\draw [->,very thick] (0,.2) -- (0,-.8);
	
	\foreach \x in {-3, -1, 1} {
		\foreach \y in {-2.7, -3.3} { \draw (\x,-3) .. controls (\x+1,\y) .. (\x+2,-3); };
	};
	\foreach \x in {-1,1} {
		\draw[dashed,red] (\x,-1.5) -- (\x, -3);
		\filldraw (\x,2.5) circle(.14);
		\filldraw (\x,-1.5) circle(.14);
		\filldraw (\x*3,-3) circle(.14);
		\filldraw (\x,-3) circle(.14);
	};
\end{scope}

\begin{scope}[yshift=-85mm]
	\draw (-1,2.5) .. controls (0,2.2) .. (1,2.5);
	\draw (-3,1) .. controls (0, 1.4) .. (3, 1);
	\draw[dashed,red] (-3,1) .. controls (0, .6) .. (3, 1);
	\filldraw (3,1) circle(.14);
	\filldraw (-3,1) circle(.14);
	
	\draw [->,very thick] (0,.2) -- (0,-.8);
	
	\foreach \y in {-2.7, -3.3} { \draw (-1,-3) .. controls (0,\y) .. (1,-3); };
	\foreach \x in {-3, 1} {
		\draw (\x,-3) .. controls (\x+1,-2.7) .. (\x+2,-3);
		\draw[dashed,red] (\x,-3) .. controls (\x+1,-3.3) .. (\x+2,-3);
		\filldraw (\x,-3) circle(.14);
	};
	\foreach \x in {-1,1} {
		\filldraw (\x,2.5) circle(.14);
		\filldraw (\x,-1.5) circle(.14) -- (\x, -3);
	};
	\filldraw (3,-3) circle(.14);
	\filldraw (-1,-3) circle(.14);
\end{scope}

\begin{scope}[xshift=105mm, yshift=-85mm]
	\draw[dotted,thick,blue] (-1,2.5) .. controls (0,2.2) .. (1,2.5);
	\draw (-3,1) .. controls (0, 1.4) .. (3, 1);
	\draw[dashed,red] (-3,1) .. controls (0, .6) .. (3, 1);
	\filldraw (-3,1) circle(.14);
	\filldraw (3,1) circle(.14);
	
	\draw [->,very thick] (0,.2) -- (0,-.8);
	
	\foreach \y in {1,-3} {
		\foreach \dy in {-1,0,1} { \draw (-3,\y) -- (-4,\y+\dy); };
	};
	\foreach \y in {1,-3} {
		\foreach \dy in {-1,0,1} { \draw[dashed,red] (3,\y) -- (4,\y+\dy); };
	};
	
	\foreach \x in {-3, -1} { \draw[dotted,thick,blue] (\x,-3) .. controls (\x+1,-3.3) .. (\x+2,-3); };
	\foreach \x in {-1, 1} { \draw[dotted,thick,blue] (\x,-3) .. controls (\x+1,-2.7) .. (\x+2,-3); };
	\draw (-3,-3) .. controls (-2,-2.7) .. (-1,-3);
	\draw[dashed,red] (1,-3) .. controls (2,-3.3) .. (3,-3);
	
	\foreach \x in {-1,1} {
		\draw[dotted,thick,blue] (\x,-1.5) -- (\x, -3);
		\filldraw (\x,2.5) circle(.14);
		\filldraw (\x,-1.5) circle(.14);
	};
	\foreach \x in {-3, -1, 1, 3} {
		\filldraw (\x,-3) circle(.14);
	};
\end{scope}

\end{tikzpicture}

\caption{Extending the coloring of $G'$ to $G$ in the proof of Theorem~\ref{le:5regDoubleEdgesAndLoops}~(\ref{noDouble3Path}). } \label{fig:noDoub}
\end{figure}
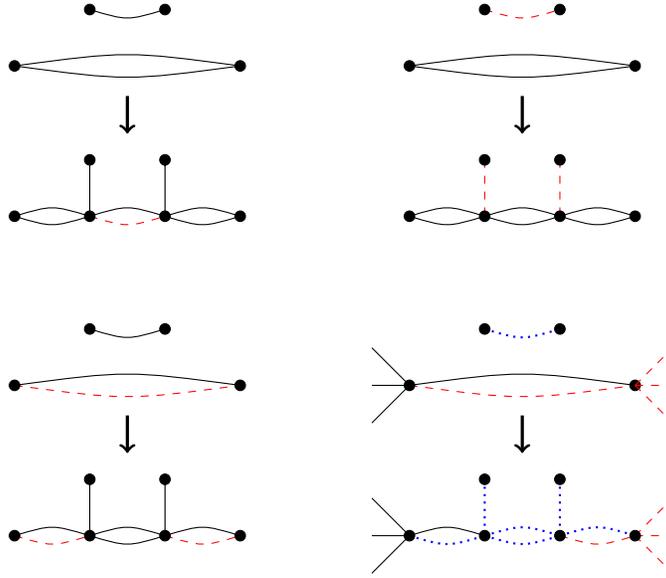

First, suppose both edges between $v_1$ and $v_4$, as well as the edge $u_2 u_3$, have color~$1$. 
Then in $G$, we give color~$1$ to all edges incident to $v_2$ or $v_3$ except for one of the edges between $v_2$ and $v_3$, to which we give color~$2$.

Second, suppose the edges between $v_1$ and $v_4$ have color~$1$ and the edge $u_2 u_3$ has color~$2$. 
Then in $G$, we give color~$2$ to $u_2 v_2$ and $u_3 v_3$ and color~$1$ to all other edges incident to $v_2$ or $v_3$.

Third, suppose $u_1 u_2$ and one of the edges between $v_1$ and $v_4$ have color~$1$, while the other has color~$2$. 
Then in $G$ we give color~$2$ to one of the edges between $v_1$ and $v_2$ and to one of the edges between $v_3$ and $v_4$.  We give color~$1$ to all other edges incident to $v_2$ or $v_3$.

Fourth, suppose all three edges have different colors. 
Then there are two subcases to consider.  
If one of the edges between $v_1$ and $v_4$ is the only edge of its color that is incident to both $v_1$ and $v_4$, then we may instead give it the same color as $u_2 u_3$ and so reduce the problem to the previous case.  
Assume, then, that the edges between $v_1$ and $v_4$ have colors $1$ and~$2$, that $v_1$ is incident to four edges with color~$1$, that $v_4$ is incident to four edges with color~$2$, and that $u_2 u_3$ has color~$3$.  
When we color $G$, we give color~$1$ to one of the edges between $v_1$ and $v_2$, color~$2$ to one of the edges between $v_3$ and $v_4$, and color~$3$ to all other edges incident to $v_2$ or $v_3$.

We have shown that in all four cases, we may extend a $(4, 1)$-coloring of~$G'$ to a $(4, 1)$-coloring of~$G$, which is a contradiction. 

\medskip

(\ref{noLoopDoubleIncidence}) Suppose to the contrary that $u$ is a vertex with a loop and that there is a double edge between $u$ and some other vertex~$v$.  
Observe that $v$ cannot have a loop: if it did, then $u$ and $v$ would each have exactly one neighbor outside of~$\{u, v\}$.  
This is a minimal edge cut of size~$2$, which contradicts Theorem~\ref{minNo41coloring}~(\ref{no2cut}).

Thus, $u$ sends one edge to a vertex~$w$ outside of~$\{u, v\}$, while $v$ sends three, to vertices $x$, $y$, and~$z$.  
We remove $u$ and $v$ from $G$ and create a $5$-regular graph~$G'$ by adding edges~$e = wx$ and~$f = yz$.  
(As usual, we may create loops or multiple edges.)  
By hypothesis, $G'$ has a $(4, 1)$-coloring.  
If $e$ and $f$ both have color~$1$, then we may extend the coloring to $G$ by coloring all edges incident to $u$ or to $v$ with color~$1$, 
 except for one edge between $u$ and $v$, to which we give color~$2$.  
If $e$ has color~$1$ and $f$ has color~$2$, 
 then we extend the coloring to $G$ by giving color~$1$ to both~$uw$ and~$vx$ and color~$2$ to all other edges incident to $u$ or $v$.  
In either case, we have a contradiction. 

\medskip

(\ref{noLoopAdjacence}) Suppose that $u$ and $v$ are adjacent vertices with loops.  By part~(\ref{noLoopDoubleIncidence}), there is exactly one edge between $u$ and $v$.
Observe that neither $u$ nor $v$ can have two loops.
Indeed, if both have two loops, then $G[\{u, v\}]$ is a component and obviously has a $(4, 1)$-coloring. 
If, say, $v$ has two loops but $u$ has only~one, then $u$ sends two edges to vertices outside of~$\{u, v\}$.  
These edges form a minimal edge cut of size~$2$, which contradicts Theorem~\ref{minNo41coloring}~(\ref{no2cut}).

Thus, we may assume that $u$ and $v$ are incident to only one loop each and hence both send two edges to vertices outside of~$\{u, v\}$.
Delete $u$ and $v$ and form a new $5$-regular graph~$G'$ by adding a matching between the (remaining) neighborhood of~$u$ and the (remaining) neighborhood of~$v$. 
By hypothesis, $G'$ has a $(4, 1)$-coloring.
Then we obtain a $(4, 1)$-coloring of $G$ much as in part~(\ref{3parallel}), a contradiction. 

\medskip

(\ref{5LoopMin}) If $X$ and $Y$ are disjoint subsets of~$V(G)$, let $e(X, Y)$ denote the number of edges between $X$ and $Y$.
Since $G$ does not admit a $(4,1)$-coloring, it does not contain a perfect matching. 
This means that there is a set~$S \subset V(G)$ such that the number of components of $G - S$ of odd order is strictly greater than~$|S|$. 
Let $C_1$, \ldots, $C_t$ be the components of $G - S$ of of odd order. 
Then $5|V(C_i)|=2|E(C_i)|+e(S,C_i)$ and hence $e(S,C_i)$ is odd, $1\leq i\leq t$.
Similarly, there is an even number of edges between $S$ and a component of $G - S$ of even order.
Therefore $|S| \equiv t \pmod 2$, and thus $t \geq |S| + 2$.

Recall from Theorem~\ref{minNo41coloring}~(\ref{special1cut},~\ref{special3cut}) that if $G$ contains a bridge, 
then one of the endpoints of the bridge is incident to two loops, and if $G$ contains a minimal edge cut of size~$3$, then its edges share an endpoint which is incident to a loop. 
Therefore, either $e(S,C_i) \geq 5$ or there is only~one vertex in $C_i$ and this vertex is adjacent to $\ell\geq 1$ loops.
In the latter case, $e(S, C_i) + 2\ell = 5$.
Let $C = \bigcup_{i=1}^t C_i$ and let $k$ be the total number of loops in $G$.
Then
\[
5t \leq \bigl\lvert e(S, C) \big\rvert + 2k \leq 5|S| +2k
\]
and hence
\[
k \geq \frac{5}{2}(t - |S|) \geq \frac{5}{2} \cdot 2 = 5,
\]
as desired. 

\medskip

(\ref{noK33withLoops}) By part~(\ref{no4regular}), we may assume that the $v_i$ form an independent set, because if, say, $v_1v_2$ were an edge in $G$, then $\{v_1, v_2, u_1, u_2, u_3\}$ would induce a $4$-regular subgraph of $G$. 
Delete all of the $u_i$ and the $v_i$ and form a new $5$-regular graph~$G'$ by adding a matching~$M = \{e_{12}, e_{23}, e_{31}\}$ among the neighborhoods of the $v_i$ such that each edge~$e_{ij}$ (which may be a loop) has one endpoint in $N(v_i)$ and the other in $N(v_j)$.

By hypothesis, $G'$ has a $(4, 1)$-coloring~$c$.  
When we extend this coloring to $G$, we will give $c(e_{ij})$ to one edge incident to $v_i$ and to one edge incident to~$v_j$.  
Furthermore, we will give to each vertex~$v_i$ an ordered triple~$(a_1,a_2,a_3) := (c(v_i u_1), c(v_i u_2), c(v_i u_3))$.
There are three cases we must consider (see Figure~\ref{fig:noK33}).

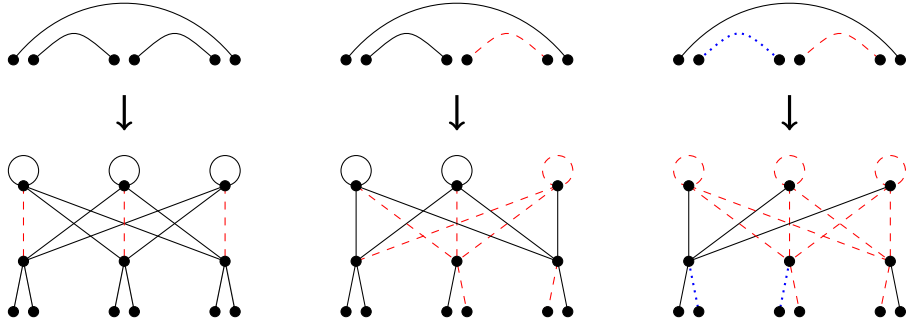
\begin{figure}[ht]
\centering
\begin{tikzpicture}[scale=.67]

\begin{scope}
	\draw (-2.2,1) .. controls (-1,2.5) and (1, 2.5) .. (2.2,1);
	\foreach \c in {-1,1} { \draw (1.8*\c, 1) .. controls (1*\c,1.7) .. (.2*\c, 1); };
	\foreach \x in {-2.2, -1.8, -.2, .2, 1.8, 2.2} { \filldraw (\x,1) circle(1mm); };
	
	\draw [->,very thick] (0,.3) -- (0,-.4);
	
	\foreach \x in {0, 2} {
		\draw (\x,-1.5) -- (-2, -3);
	};
	\foreach \x in {-2, 2} {
		\draw (\x,-1.5) -- (0, -3);
	};
	\foreach \x in {-2, 0} {
		\draw (\x,-1.5) -- (2, -3);
	};
	\foreach \x in {-2, 0, 2} {
		\draw[dashed,red] (\x,-1.5) -- (\x,-3);
		\draw (\x,-1.2) circle(3mm);
		\foreach \y in {-1.5,-3} { \filldraw (\x,\y) circle(1mm); };
		\foreach \dx in {-.2, .2} { \filldraw (\x, -3) -- (\x+\dx, -4) circle(1mm); };
	};
	
\end{scope}

\begin{scope}[xshift=66mm]
	\draw (-2.2,1) .. controls (-1,2.5) and (1, 2.5) .. (2.2,1);
	\draw (-1.8, 1) .. controls (-1,1.7) .. (-.2, 1);
	\draw[dashed,red] (1.8, 1) .. controls (1,1.7) .. (.2, 1);
	\foreach \x in {-2.2, -1.8, -.2, .2, 1.8, 2.2} { \filldraw (\x,1) circle(1mm); };
	
	\draw [->,very thick] (0,.3) -- (0,-.4);
	
	\draw (-2, -1.2) circle(3mm);
	\draw[dashed,red] (2,-1.5) -- (-2,-3);
	\draw (-2.2, -4) -- (-2,-3) -- (-1.8,-4);
	\draw (0, -1.2) circle(3mm);
	\draw[dashed,red] (0,-1.5) -- (0,-3) -- (.2, -4);
	\draw (0,-3) -- (-.2,-4);
	\draw[dashed,red] (2, -1.2) circle(3mm);
	\draw (2,-1.5) -- (2,-3) -- (2.2, -4);
	\draw[dashed,red]  (2,-3) -- (1.8,-4);
	
	\foreach \x in {0, -2} {
		\draw (\x,-1.5) -- (-2, -3);
	};
	\foreach \x in {-2, 2} {
		\draw[dashed,red] (\x,-1.5) -- (0, -3);
	};
	\foreach \x in {-2, 0} {
		\draw (\x,-1.5) -- (2, -3);
	};
	\foreach \x in {-2, 0, 2} {
		\foreach \y in {-1.5,-3} { \filldraw (\x,\y) circle(1mm); };
		\foreach \dx in {-.2, .2} { \filldraw (\x+\dx, -4) circle(1mm); };
	};
\end{scope}

\begin{scope}[xshift=132mm]
	\draw (-2.2,1) .. controls (-1,2.5) and (1, 2.5) .. (2.2,1);
	\draw[dotted,thick,blue] (-1.8, 1) .. controls (-1,1.7) .. (-.2, 1);
	\draw[dashed,red] (1.8, 1) .. controls (1,1.7) .. (.2, 1);
	\foreach \x in {-2.2, -1.8, -.2, .2, 1.8, 2.2} { \filldraw (\x,1) circle(1mm); };
	
	\draw [->,very thick] (0,.3) -- (0,-.4);
	
	\draw[dashed,red] (-2, -1.2) circle(3mm);
	\draw (-2,-1.5) -- (-2,-3) -- (-2.2, -4);
	\draw[dotted,thick,blue] (-2,-3) -- (-1.8,-4);
	\draw[dashed,red] (0, -1.2) circle(3mm);
	\draw[dashed,red] (0,-1.5) -- (0,-3) -- (.2, -4);
	\draw[dotted,thick,blue] (0,-3) -- (-.2,-4);
	\draw[dashed,red] (2, -1.2) circle(3mm);
	\draw[dashed,red] (2,-1.5) -- (2,-3) -- (1.8, -4);
	\draw (2,-3) -- (2.2,-4);
	
	\foreach \x in {0, 2} {
		\draw (\x,-1.5) -- (-2, -3);
	};
	\foreach \x in {-2, 2} {
		\draw[dashed,red] (\x,-1.5) -- (0, -3);
	};
	\foreach \x in {-2, 0} {
		\draw[dashed,red] (\x,-1.5) -- (2, -3);
	};
	\foreach \x in {-2, 0, 2} {
		\foreach \y in {-1.5,-3} { \filldraw (\x,\y) circle(1mm); };
		\foreach \dx in {-.2, .2} { \filldraw (\x+\dx, -4) circle(1mm); };
	};
\end{scope}

\end{tikzpicture}

\caption{Extending the coloring of $G'$ to $G$ in the proof of Theorem~\ref{le:5regDoubleEdgesAndLoops}~(\ref{noK33withLoops}). } \label{fig:noK33}
\end{figure}

First, suppose that all of the $e_{ij}$ have color~$1$. 
In this case, we give $v_1$ the triple~$(2,1,1)$, $v_2$ the triple~$(1,2,1)$, and $v_3$ the triple~$(1,1,2)$. 
Additionally, we give color~$1$ to the loop at each $u_i$.

Second, suppose that the $e_{ij}$ have exactly two colors. 
Without loss of generality, let $c(e_{12}) = c(e_{31}) = 1$ and $c(e_{23}) = 2$.  
Observe that in $G$, $v_1$ is incident to two edges with color~$1$, while $v_2$ and $v_3$ are each incident to one edge with color~$1$ and one edge with color~$2$.  
We give the triple~$(1, 1, 2)$ to~$v_1$, $(2,2,2)$ to~$v_2$, and~$(1,1,1)$ to~$v_3$.  
Additionally, we give color~$1$ to the loops at $u_1$ and $u_2$ and color~$2$ to the loop at $u_3$.

Third, suppose that all of the $e_{ij}$ have different colors.  
Without loss of generality, let $c(e_{12}) = 1$, $c(e_{23}) = 2$, and $c(e_{31}) = 3$. 
In this case, we give the triple~$(1,1,1)$ to~$v_1$ and~$(2,2,2)$ to both $v_2$ and $v_3$. 
Additionally, we give color~$2$ to the loop at each $u_i$.

In all three cases, we have produced a $(4, 1)$-coloring of~$G$, which is a contradiction.

\medskip

(\ref{no4LoopNeighbors}) Suppose that $u \in V(G)$ is adjacent to vertices $v_1$, $v_2$, $v_3$, and~$v_4$, all of which have loops.  
By part~(\ref{noLoopAdjacence}), the $v_i$ form an independent set.  
By part~(\ref{noLoopDoubleIncidence}), there is only one edge between $u$ and each $v_i$.  
Delete the $v_i$ and form a new $5$-regular graph $G'$ by adding two loops at $u$, and, for each $i$ such that $v_i$ has only one loop, adding an edge~$e_i$ between the two vertices of $N(v_i) \setminus \{u\}$.  
By hypothesis, $G'$ has a $(4, 1)$-coloring.  
We extend this coloring to~$G$ as follows: for each $v_i$ with only one loop, we give all edges incident to $v_i$, except for $u v_i$, the same color as $e_i$.  
We then give each $u v_i$ the color of the loops incident to $u$ in $G'$, which we may assume is a new color.  
Finally, if any of the $v_i$ have two loops, we give these loops a color different from the color of~$u v_i$. 
Thus, $G$ has a $(4, 1)$-coloring, which is a contradiction.

\medskip

(\ref{Forbidden4}) The proof of this statement is computationally assisted but can be checked by hand with extensive case work. 
An exhaustive search shows that there exist only seven graphs on $4$ vertices with at least~$8$ edges and with maximum degree~$5$ that satisfy parts (\ref{no4regular}-\ref{noLoopAdjacence}) of this theorem (see Figure~\ref{fig:Forbidden4}).
Moreover, all seven graphs have exactly $8$ edges.

Let $H$ be a subgraph of $G$ with $4$ vertices and $8$ edges. 
Since $G$ is $5$-regular, there are $4$ edges between $H$ and $G-H$. 
Let $U = \{u,v,w,x\}$ be the multiset of vertices in $H$, 
where the multiplicity of a vertex in $U$ equals the number of edges between the vertex and $G-H$. 
Let $U' = \{u',v',w',x'\}$ be the corresponding neighbors in $G-H$.

In each of the seven present graphs, 
$H$ has a $1$-factor $M$ and there are two distinct vertices in $U$, $u$ and $v$ (without loss of generality), so that $H-\{u,v\}$ has a non-loop edge $e$.
Define the $5$-regular graph $G' = (G-H) + \{u'v',w'x'\}$, which has fewer vertices than $G$.
By our assumption of vertex-minimality, $G'$ has a $(4,1)$-coloring~$c$.
We can then define a $(4,1)$-coloring of $G$ as follows.
Let $c_1 = c(u'v')$ and $c_2 = c(w'x')$.
If $c_1 = c_2$, use a new color for the edges in the one-factor $M$ and use $c_1$ for all other edges incident to a vertex in $H$.
If $c_1 \neq c_2$, use $c_1$ for $uu'$, $vv'$, and $e$ and use $c_2$ for all other edges incident to a vertex in $H$. \qedhere

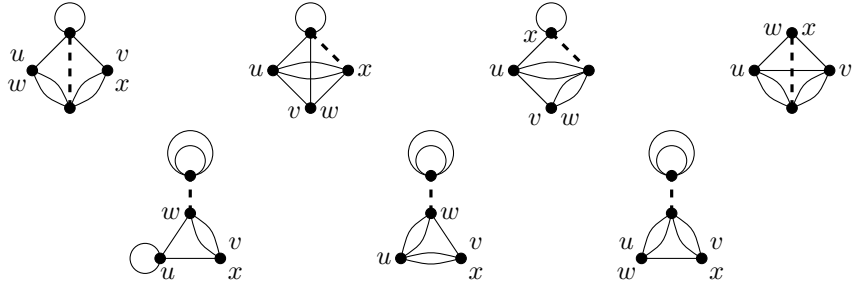
\begin{figure}[ht]
	\centering

	\begin{tikzpicture}
	
		\begin{scope}[xshift=16mm, yshift=-19mm]
			\draw[dashed,very thick] (0,0) -- (0,.5);
			\draw (0,.8) circle(.3);
			\draw (0,.7) circle(.2);
			\foreach \w in {-1,1}{
				\draw (.4,-.6) .. controls(.2-.1*\w,-.3-.1*\w) .. (0,0);
			}
			\draw (0,0) -- (-.4,-.6) -- (.4,-.6);
			\draw (-.6,-.6) circle(.2);
			\draw (.6,-.6) node[above]{$v$} node[below]{$x$};
			\filldraw (-.4,-.6) circle(.07);
			\draw (-.3,-.6) node[below]{$u$};
			\filldraw (.4,-.6) circle(.07);
			\filldraw (0,0) circle(.07) node[left]{$w$};
			\filldraw (0,.5) circle(.07);
		\end{scope}
	
		\begin{scope}[xshift=48mm, yshift=-19mm]
			\draw[dashed,very thick] (0,0) -- (0,.5);
			\draw (0,.8) circle(.3);
			\draw (0,.7) circle(.2);
			\foreach \w in {-1,1}{
				\draw (-.4,-.6) .. controls(-.2+.1*\w,-.3-.1*\w) .. (0,0);
				\draw (-.4,-.6) .. controls(0,-.6+.1*\w) .. (.4,-.6);
			}
			\draw (0,0) -- (.4,-.6);
			\draw (.6,-.6) node[above]{$v$} node[below]{$x$};
			\filldraw (-.4,-.6) circle(.07) node[left]{$u$};
			\filldraw (.4,-.6) circle(.07);
			\filldraw (0,0) circle(.07) node[right]{$w$};
			\filldraw (0,.5) circle(.07);
		\end{scope}
	
		\begin{scope}[xshift=80mm, yshift=-19mm]
			\draw[dashed,very thick] (0,0) -- (0,.5);
			\draw (0,.8) circle(.3);
			\draw (0,.7) circle(.2);
			\foreach \w in {-1,1}{
				\draw (-.4,-.6) .. controls(-.2+.1*\w,-.3-.1*\w) .. (0,0);
				\draw (.4,-.6) .. controls(.2+.1*\w,-.3+.1*\w) .. (0,0);
			}
			\draw (-.4,-.6) -- (.4,-.6);
			\draw (-.6,-.6) node[above]{$u$} node[below]{$w$};
			\draw (.6,-.6) node[above]{$v$} node[below]{$x$};
			\filldraw (-.4,-.6) circle(.07);
			\filldraw (.4,-.6) circle(.07);
			\filldraw (0,0) circle(.07);
			\filldraw (0,.5) circle(.07);
		\end{scope}
	
		\foreach \w in {-.5,.5} {
			\draw (-.5,0) .. controls(-.25-.2*\w,-.25-.2*\w) .. (0,-.5);
			\draw (.5,0) .. controls(.25+.2*\w,-.25-.2*\w) .. (0,-.5);
			\filldraw (0,.5) -- (\w,0) circle(.07);
			\filldraw (0,\w) circle(.07);
		};
		\draw (0,.7) circle(.2);
		\draw[dashed,very thick] (0,-.5) -- (0,.5);
		\draw (-.7,0) node[above]{$u$} node[below]{$w$};
		\draw (.7,0) node[above]{$v$} node[below]{$x$};
		
		\begin{scope}[xshift=32mm]
			\foreach \w in {-.5,.5} {
				\draw (-.5,0) .. controls(0,.3*\w) .. (.5,0);
				\draw (0,-.5) -- (\w,0);
				\filldraw (0,\w) circle(.07);
				\filldraw (\w,0) circle(.07);
			};
			\draw (0,.7) circle(.2);
			\draw (0,-.5) -- (0,.5) -- (-.5,0);
			\draw[dashed,very thick] (0,.5) -- (.5,0);
			\draw (-.5,0) node[left]{$u$};
			\draw (.5,0) node[right]{$x$};
			\draw (0,-.55) node[left]{$v$} node[right]{$w$};
		\end{scope}
		
		\begin{scope}[xshift=64mm]
			\foreach \w in {-.5,.5} {
				\draw (-.5,0) .. controls(0,.3*\w) .. (.5,0);
				\draw (0,-.5) .. controls(.25+.2*\w,-.25-.2*\w) .. (.5,0);
				\filldraw (-.5,0) -- (0,\w) circle(.07);
				\filldraw (\w,0) circle(.07);
			};
			\draw (0,.7) circle(.2);
			\draw[dashed,very thick] (0,.5) -- (.5,0);
			\draw (-.05,.45) node[left]{$x$};
			\draw (-.5,0) node[left]{$u$};
			\draw (0,-.65) node[left]{$v$} node[right]{$w$};
		\end{scope}
		
		\begin{scope}[xshift=96mm]
			\foreach \w in {-.5,.5} {
				\draw (-.5,0) .. controls(-.25-.2*\w,-.25-.2*\w) .. (0,-.5);
				\draw (.5,0) .. controls(.25+.2*\w,-.25-.2*\w) .. (0,-.5);
				\filldraw (0,.5) -- (\w,0) circle(.07);
				\filldraw (0,\w) circle(.07);
			};
			\draw (-.5,0) -- (.5,0);
			\draw[dashed,very thick] (0,-.5) -- (0,.5);
			\draw (-.5,0) node[left]{$u$};
			\draw (.5,0) node[right]{$v$};
			\draw (0,.55) node[left]{$w$} node[right]{$x$};
		\end{scope}
		
	\end{tikzpicture}

	\caption{$4$-vertex, $8$-edge graphs with the vertices in $U$ labeled and an edge in $H - \{u,v\}$ dashed, from the proof of Theorem~\ref{le:5regDoubleEdgesAndLoops}~(\ref{Forbidden4}).} \label{fig:Forbidden4}
\end{figure}

\end{proof}

\subsection{\texorpdfstring{$5$}{5}-regular graphs without  \texorpdfstring{$\{4,1\}$}{\{4,1\}}-factors}\label{se:no-4-1-factor}

The results in this subsection are very similar to those in the previous subsection, so we will omit some of the proofs.
Notice first that every statement of Theorem~\ref{minNo41coloring} also holds for vertex-minimal graphs without $\{4,1\}$-factors because the proofs do not require the use of more than two colors.  

\begin{theo}\label{minNo41factor}
 Let $G$ be a vertex-minimal $5$-regular graph without a $\{4,1\}$-factor.
 \begin{enumerate}[(a)]
  \item $G$ is connected.
  \item $G$ is not $4$-edge-connected, i.e., contains an edge cut on $3$ edges.
  \item $G$ has no minimal edge cut of size $2$.
  \item $G$ does not have two bridges. \label{factorNo2bridges}
  \item Each bridge in $G$ has (precisely) one endpoint incident to two loops. \label{factorSpecial1cut}
  \item The edges of any minimal edge cut of size $3$ in $G$ have a vertex in common, and this vertex is incident to a loop. \hfill \qed
 \end{enumerate}
\end{theo} 

Most of the statements in Theorem~\ref{le:5regDoubleEdgesAndLoops} also hold for vertex-minimal graphs without $\{4,1\}$-factors.  We discuss the differences between Theorems \ref{le:5regDoubleEdgesAndLoops} and~\ref{le:5regDEAL-factor} in Remark~\ref{rk:5regFactorConditions} below.

\begin{theo} \label{le:5regDEAL-factor}
Let $G$ be a vertex-minimal $5$-regular graph without a $\{4,1\}$-factor.
\begin{enumerate}[(a)]
  \item $G$ does not contain a copy of~$K_4$. \label{noK4}
\item $G$ has no parallel non-loop edges. \label{noDouble}
\item No vertices with loops are adjacent.  \label{noLoopAdjacence-factor}
\item $G$ contains at least 5 loops. \label{5LoopMin-factor}
\item There do not exist $u_1$, $u_2$, $u_3$, $v_1$, $v_2$, $v_3 \in V(G)$ such that the $u_i$ have loops and such that for each $i$ and $j$, $u_i$ is adjacent to $v_j$. \label{noK33withLoops-factor}
\end{enumerate}
\end{theo}

\begin{remark}\label{rk:5regFactorConditions}
Here, we elaborate on the relationships between the statements in Theorems \ref{le:5regDoubleEdgesAndLoops} and~\ref{le:5regDEAL-factor}.  First, the proofs of Theorem~\ref{le:5regDoubleEdgesAndLoops}~(\ref{no4regular}, \ref{no4LoopNeighbors}) do not work for factors, since in each case, we may need three colors to create the contradictory $(4,1)$-coloring of~$G$.

Next, the proof of Theorem~\ref{le:5regDEAL-factor} (\ref{noK4}) given below does not work for general $(4, 1)$-colorings, because there might be three edges of color~$1$ and one edge of color~$2$ incident to one endpoint of the new edge and with four edges of color~$3$ incident to the other endpoint.  (This corresponds to the last configuration in Figure~\ref{K4removed}, but with a third color assigned to the four lower edges.)  It is not hard to show that it is impossible to extend this coloring to a $(4, 1)$-coloring of the original graph.

Next, we can improve on the condition of Theorem~\ref{le:5regDoubleEdgesAndLoops}~(\ref{3parallel}) to prohibit double edges: see Theorem~\ref{le:5regDEAL-factor}~(\ref{noDouble}).
(We cannot improve the statement for $(4, 1)$-colorings, because, if we try to follow the proof of Theorem~\ref{le:5regDEAL-factor}~(\ref{noDouble}) given below, we may obtain three different colors from the smaller graph~$G'$, making extension to a $(4,1)$-coloring of~$G$ impossible.) 
So, the analogous statements to Theorem~\ref{le:5regDoubleEdgesAndLoops}~(\ref{noDouble3Path},~\ref{noLoopDoubleIncidence}) for factors are merely special cases of forbidding parallel non-loop edges.
Similarly, with no parallel non-loop edges and, by Theorem~\ref{minNo41factor} (\ref{factorNo2bridges}, \ref{factorSpecial1cut}), at most~one double loop, the analogous statement to Theorem~\ref{le:5regDoubleEdgesAndLoops}~(\ref{Forbidden4}) is immediate.

Finally, the proofs of Theorem~\ref{le:5regDoubleEdgesAndLoops}~(\ref{noLoopAdjacence}, \ref{5LoopMin}, \ref{noK33withLoops}), which correspond to Theorem~\ref{le:5regDEAL-factor}~(\ref{noLoopAdjacence-factor}, \ref{5LoopMin-factor}, \ref{noK33withLoops-factor}), work for $\{4,1\}$-factors in exactly the same way, so we will not give the proofs.

\end{remark}

\begin{proof}[Proof of Theorem~\ref{le:5regDEAL-factor}.]

(\ref{noK4}) Let $K=\{u_1, u_2, v_1, v_2\}$ denote the vertices of a copy of~$K_4$ in $G$.
  We obtain a graph~$G'$ by removing all vertices in $K$ from $G$ and adding two adjacent new vertices $u$ and~$v$.  Then, for each $x \notin K$, we add an edge~$xu$ for each edge~$xu_i$ in $G$ and an edge~$xv$ for each edge~$xv_i$ in $G$, $i\in\{1,2\}$.  (Note that this may create multiple edges.)

The new graph~$G'$ is $5$-regular and has fewer vertices than $G$. By the assumption of vertex-minimality, it has a $\{4,1\}$-factor.  
This $\{4,1\}$-factor extends to a $\{4,1\}$-factor of~$G$ regardless of the colors of the edges incident to $u$ and $v$ (see Figure~\ref{K4removed}). 
This is a contradiction.
    
  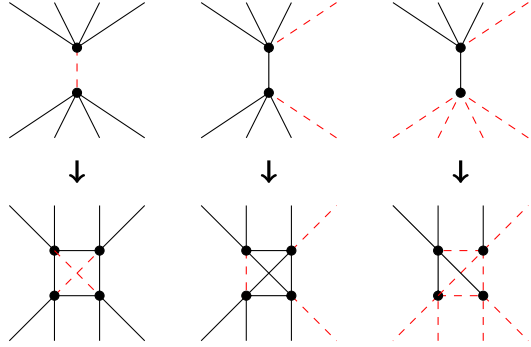
\begin{figure}[h!tb]
   \centering
   
   \begin{tikzpicture}[scale=.3]
   	\begin{scope}
		\foreach \x in {-3,-1,1,3} { 
			\draw (\x,8) -- (0,6); 
			\draw (\x,2) -- (0,4); 
		};
		\draw[dashed,red] (0,6) -- (0,4);
		\filldraw (0,6) circle(2mm);
		\filldraw (0,4) circle(2mm);
		
		\draw [->,very thick] (0,1) -- (0,0);
		
		\foreach \c in {-1,1} {
			\foreach \x in {1,3} { 
				\draw (\c*1,-3) -- (\c*\x,-1);
				\draw (\c*1,-5) -- (\c*\x,-7);
			}; 
			\draw (-1,\c-4) -- (1,\c-4);
			\draw[dashed,red] (-1,-4+\c) -- (1,-4-\c);
			\filldraw (\c,-3) circle(2mm) -- (\c,-5) circle(2mm);
		};
		
	\end{scope}
	
   	\begin{scope}[xshift=85mm]
		\foreach \x in {-3,-1,1} { 
			\draw (\x,8) -- (0,6); 
			\draw (\x,2) -- (0,4); 
		};
		
		\draw[dashed,red] (3,8) -- (0,6);
		\draw[dashed,red] (3,2) -- (0,4);
		\filldraw (0,6) circle(2mm) -- (0,4) circle(2mm);
		
		\draw [->,very thick] (0,1) -- (0,0);
		
		\draw[dashed,red] (-1,-3) -- (-1,-5) ;
		\draw (1,-3) -- (1,-5) ;
		\foreach \c in {-1,1} {
			\draw (\c*1,-3) -- (\c*1,-1);
			\draw (\c*1,-5) -- (\c*1,-7);
			\draw (-1,-4+\c) -- (-3,-4+\c*3);
			\draw[dashed,red] (1,-4+\c) -- (3,-4+\c*3);
			\draw (-1,-4+\c) -- (1,-4-\c);
			\filldraw (-1,\c-4) circle(2mm) -- (1,\c-4) circle(2mm);
		};
		
	\end{scope}
	
   	\begin{scope}[xshift=17cm]
		\foreach \x in {-3,-1,1} { 
			\draw (\x,8) -- (0,6); 
			\draw[dashed,red] (\x,2) -- (0,4); 
		};
		\draw[dashed,red] (3,8) -- (0,6);
		\draw[dashed,red] (3,2) -- (0,4);
		\filldraw (0,6) circle(2mm) -- (0,4) circle(2mm);
		
		\draw [->,very thick] (0,1) -- (0,0);
		
		\draw[dashed,red] (1,-3) -- (3,-1) ;
		\foreach \c in {-1,1} {
			\draw (\c*1,-3) -- (\c*1,-1);
			\draw[dashed,red] (\c,-7) -- (\c,-5);
			\draw[dashed,red] (\c,-5) -- (\c*3,-7);
			\draw[dashed,red] (-1,-5) -- (1,-4+\c);
			\draw[dashed,red] (1,-3) -- (\c,-4-\c);
			\filldraw (-1,-3) -- (\c,-5) circle(2mm);
		};
		\filldraw (1,-3) circle(2mm);
		\filldraw (-3,-1) -- (-1,-3) circle(2mm) ;
		
	\end{scope}
   \end{tikzpicture}
   
   \caption{All possible configurations of a $\{4,1\}$-factor at the edge~$uv$ and corresponding $\{4,1\}$-factors using edges from the copy of~$K_4$ (up to taking complements of color classes), from the proof of Theorem~\ref{le:5regDEAL-factor} (\ref{noK4}).}
   \label{K4removed}
  \end{figure}
  
 (\ref{noDouble}) 
  Assume that there are at least two edges between $u$ and $v$.
  Consider the $5$-regular graph $G'$ obtained from $G$ by removing $u$ and $v$ and adding a matching between $N(u) \setminus \{v\}$ and $N(v) \setminus \{u\}$ (possibly creating parallel edges and loops).
  By assumption, $G'$ has a $\{4,1\}$-factor~$F$.
  We can extend $F$ to a $\{4,1\}$-factor of~$G$ by adding some of the edges between $u$ and $v$ to $F$, which is a contradiction.

\qedhere

  \end{proof}

\section{\texorpdfstring{$r$}{r}-Regular Graphs for \texorpdfstring{$r\geq 6$}{r > 5}} \label{6reg}

In this section we give a negative answer to the analogue of Question~\ref{QUE:4-1-coloring} for $r\geq 6$.
More generally, for each odd $t$ and each even $r$, as well as for each odd $t$ and each odd $r\geq (t+2)(t+1)$, we construct an $r$-regular graph with no $(r - t, t)$-coloring.  Note that for even $t$, every $r$-regular graph has a $(r-t,t)$-coloring and for odd $t\leq \frac{r}{3}$ and even $r$ every $r$-regular graph has a $(r-t,t)$-coloring due to Theorem~\ref{THM:values}.

\begin{theo}\label{theo:regular-no-coloring}
Let $r$ and $t$ be positive integers with $t\leq \frac{r}{2}$ odd.  If $r$ is even or $r\geq (t+2)(t+1)$, then there exists an $r$-regular graph that is not $(r-t,t)$-colorable.	
\end{theo}

Observe that this is the same upper bound on odd $r$ as in Theorem~\ref{THM:values}(b) (due to~\cite{AR}) for the existence of $r$-regular graphs without $\{r-t,t\}$-factors.

\begin{proof}
 First, if $r$ is even, then the $r$-regular graph with one vertex and $\frac{r}{2}$ loops has no $(r-t,t)$-coloring, since $t$ is odd.
 
Now suppose that $r\geq (t+2)(t+1)\geq 6$ is odd.
 Let $G$ be a graph on vertices $v$, $u$, $u_1$, \dots,~$u_{t+1}$ with $t+2$ edges between $v$ and $u_i$ and $\frac{r-t-2}{2}$ loops incident to $u_i$, $1\leq i\leq t+1$, and $r-(t+2)(t+1)\geq 0$ edges between $v$ and $u$ and $\frac{(t+2)(t+1)}{2}$ loops incident to $u$.
 Observe that $G$ is $r$-regular.
 Suppose that $G$ admits an $(r-t,t)$-coloring and observe that in any such coloring, there is an $i$ such that all~$t+2$ edges between $v$ and $u_i$ are of the same color.
 However, this is a contradiction, because there is no coloring of the loops incident to this $u_i$ such that there are exactly~$t$ edges of another color incident to $u_i$, as $t$ is odd.
\end{proof}

Now we will exhibit $r$-regular graphs of even order that have $(r-1,1)$-colorings but not $\{r-1,1\}$-factors. The constructions are similar to constructions in~\cite{LWY}.

\begin{theo}
 For every even $r \geq 6$ there exists an $(r-1,1)$-colorable $r$-regular graph of even order without an $\{r-1,1\}$-factor.
\end{theo}
\begin{proof}
 Note that $K_{r+1}$ has an odd number of vertices and thus does not have an $\{r-1,1\}$-factor, as $r-1$ is odd.
 However, there is an $(r-1,1)$-coloring with $3$ colors.
 Indeed color a copy of $K_{r}$ in red, $r-1$ of the remaining edges blue and the last edge green.

 If $\frac{r}{2}$ is odd, then let $G_1,\ldots, G_\frac{r}{2}$ be vertex-disjoint copies of~$K_{r+1}-e$.
 Form a graph $G$ from the union of $G_i$ by connecting all vertices of degree~$r-1$ in the $G_i$ to a new vertex $u$.
 Then $G$ has an even number of vertices and is $r$-regular.
 Moreover there is an $(r-1,1)$-coloring with $3$ colors.
 Indeed start coloring the edges incident to $u$ and extend the coloring to each $G_i$, $1\leq i\leq\frac{r}{2}$, using the coloring of~$K_{r+1}$ given above.
 Assume that $G$ has an $\{r-1,1\}$-factor, i.e., an $(r-1,1)$-coloring in two colors.
 Then there is an $i$, $1\leq i\leq\frac{r}{2}$, such that both edges between $G_i$ and $u$ are of the same color.
 This yields an $(r-1,1)$-coloring of $K_{r+1}$ in two colors, a contradiction.

 If $\frac{r}{2}$ is even, then let $t=3(\frac{r}{2}-1)$.
 Let $G_1$, \ldots, $G_t$ be vertex-disjoint copies of $K_{r+1}-e$.
 Form a graph~$G$ from the union of the $G_i$ and a disjoint copy of~$K_3$ with vertex set~$\{u_0, u_1, u_2\}$ by connecting both vertices of degree $r-1$ in $G_i$ to $u_j$ if $j(\frac{r}{2}-1)<i\leq (j+1)(\frac{r}{2}-1)$.
 Then $G$ has an even number of vertices and is $r$-regular.
 One can show that $G$ has an $(r-1,1)$-coloring but no $\{r-1,1\}$-factor with arguments similar  to those given above.
\end{proof}

\section{Concluding Remarks}\label{se:remarks}

Here we state a number of open problems related to our work. 
Recall from the Introduction that Tashkinov~\cite{Tashkinov-4-regular} showed that every $4$-regular graph with no multiple edges and at most~one loop contains a $3$-regular subgraph. 
It is not known whether the restriction on the number of loops is necessary.

\begin{question}\label{QUE:4regNoMultiedges}
Does every $4$-regular graph with no multiple edges have a $3$-regular subgraph?
\end{question}

Let us note that Question~\ref{QUE:4regNoMultiedges} is open even for the class of $4$-regular graphs with no multiple edges and at most~two loops.

Our next question concerns $(r - 1, 1)$-colorings with a bounded number of colors. 
Bernshteyn~\cite{B} showed that if $G$ is a $4$-regular graph that has a $(3, 1)$-coloring, then $G$ has a $(3, 1)$-coloring that uses at most~three colors.

\begin{question}\label{QUE:4-1-coloring-bounded}
Is there a positive integer~$K$ such that every $5$-regular graph has a $(4, 1)$-coloring using at most~$K$ colors?
\end{question}

Question~\ref{QUE:4-1-coloring-bounded} lies ``between'' Conjecture~\ref{CONJ:4-1-factor} and Question~\ref{QUE:4-1-coloring} in the following sense.  
An affirmative answer to Question~\ref{QUE:4-1-coloring-bounded} clearly gives an affirmative answer to Question~\ref{QUE:4-1-coloring}.  
On the other hand, as observed in the Introduction, Conjecture~\ref{CONJ:4-1-factor} implies an affirmative answer to Question~\ref{QUE:4-1-coloring-bounded} with $K = 2$.  
Let us also note that none of the proofs of the statements in Theorems \ref{minNo41coloring} and~\ref{le:5regDoubleEdgesAndLoops} required more than~three colors.

Our final question concerns ordered $(r - 1, 1)$-colorings.

\begin{question}\label{QUE:ordered-coloring}
For $r \geq 5$, if $G$ is an $r$-regular graph with an $(r - 1)$-regular subgraph, does $G$ admit an ordered $(r - 1, 1)$-coloring?
\end{question}

As observed in the Introduction, the converse to this statement always holds (see Figure~\ref{Relation}).  
Also, Theorem~\ref{3-regular-ordered-coloring} implies that the corresponding statement is true for $r = 4$.

\section*{Acknowledgments}\label{se:ack}

We are grateful to Maria Axenovich, Sogol Jahanbekam, Yunfang Tang, Claude Tardif, and Torsten Ueckerdt for helpful conversations.


\section*{References}

\bibliographystyle{plain}
\bibliography{SSL-BIB}
\end{document}